\documentclass[11pt]{amsart}
\usepackage{amsfonts}
\usepackage{amsmath}
\usepackage{amssymb}
\usepackage{graphicx}

\theoremstyle{definition}
\newtheorem{definition}{Definition}[section]
\theoremstyle{plain}
\newtheorem{lemma}[definition]{Lemma}

\newtheorem{theorem}[definition]{Theorem}
\newtheorem{corollary}[definition]{Corollay}
\newtheorem{example}[definition]{Example}

\numberwithin{equation}{section}

\begin{document}
\title[A New Darbo's Theorem by The Sequences of Functions]{A New Type of
Darbo's Fixed Point Theorem Defined by The Sequences of Functions }
\author[V. Karakaya]{Vatan Karakaya}
\address[V. Karakaya]{Department of Mathematical Engineering, Yildiz
Technical University, Davutpasa Campus, Esenler, 34210 Istanbul,Turkey}
\email{\texttt{vkkaya@yahoo.com}}
\author[N. \c{S}im\c{s}ek]{Necip \c{S}im\c{s}ek}
\address[N. \c{S}im\c{s}ek]{ Department of Mathematics, Istanbul Commerce
University, S\"{u}tl\"{u}ce Campus, Beyoglu, 34445 Istanbul,Turkey}
\email{necsimsek@yahoo.com}
\author[D. Sekman]{Derya Sekman}
\address[D. Sekman]{ Department of Mathematics, Ahi Evran University, Ba\u{g}%
ba\c{s}\i\ Campus, 40100 K\i r\c{s}ehir,Turkey}
\email{deryasekman@gmail.com}

\begin{abstract}
In this paper, we introduce a new type of Darbo's fixed point theorem by
using concept of function sequences with shifting distance property.
Afterward, we investigate existence of fixed point under this the theorem.
Also we are going to give interesting example held the conditions of
sequences of functions.
\end{abstract}

\maketitle

\section{Introduction}

Fixed point theory, which develops as a subbranch of operator theory, is
closely related to application areas of mathematics and the various
disciplines such as the geometry of Banach spaces, mesaure of
noncompactness, game theory and economy, so on. The topological aspect of
the fixed point theory, which proceeds in topological and metric terms, is
based on the Brouwer's fixed point theorem. The Brouwer's fixed point
theorem defined on finite dimensional spaces is generalized to infinite
dimensional spaces under the compact operator by Schauder \cite{schauder}.
However, the Schauder fixed point theorem is insufficient for noncompact
operators. In addition for noncompact operators, the existence of fixed
point is obtained by Darbo's fixed point theorem \cite{darbofixed}. The
definition of measure of noncompactness was introduced by Kuratowski and
Hausdorff \cite{kuro}. Many mathematicians have used the measure of
noncompactness concept and the Darbo's fixed point theorem to solve the
integral, differential and functional equation classes and then they have
achieved significant results \cite{banas1,banas2,mursaleen}.

The most effective and useful tools of fixed point theory are the concepts
of properties in contraction mappings classes. The first of this mapping
types is the Banach contraction principle \cite{banach}. This principle is
used to find existence and uniqueness of solution for a class of linear and
nonlinear equation systems and it has been generalized and extended by many
authors under different conditions, references therein \cite%
{derya,bouzara,karakaya}. In 1984, Khan et al. \cite{khan} introduced the
concept of altering distance functions and obtained some results on the
uniqueness of fixed point in complete metric space. Rhoades \cite{rhoades}
extended and generalized this concept to complete metric space and proved a
generalized this result by Dutta and Choudhury \cite{dutta}. Later, Berzig 
\cite{berzig} introduced the concept of shifting distance functions and
established fixed point theorem which generalized Banach contraction
principle. Recently, Samadi and Ghaemi \cite{samadi} prove some
generalizations of Darbo's fixed point theorem associated with measure of
noncompactness by using the notion of shifting distance functions and given
an application of the integral equation of mixed type.

In this work, we aim to contribute to functional analysis and operator
theory by making a generalization of Darbo's fixed point theorem with the
help of function sequences. We introduce a new type of Darbo's fixed point
theorem by using concept of function sequences with shifting distance
property. Afterward, we investigate existence of fixed point under this the
theorem. Also we are going to give interesting example held the conditions
of sequences of functions.

\section{Preliminaries}

Let $E$ be a nonempty subset of a Banach space $X$. We define $\overline{E}$
and $Conv(E)$ the closure and closed convex hull of $E,$ respectively. Also,
we denote by $M_{X}$ which is the family of all nonempty bounded subsets of $%
X$ and $N_{X}$ that is subfamily consisting of all relatively compact
subsets of $X$.

\begin{definition}[see; \protect\cite{darbo}]
\label{darbo}A mapping $\mu :M_{X}\rightarrow 
\mathbb{R}
^{+}$ is called a measure of noncompactness if it satisfy the following
conditions

$(M_{1})$ The family Ker $\mu =\left\{ A\in M_{X}:\mu (A)=0\right\} $ is
nonempty and Ker $\mu \subseteq N_{X}$

$(M_{2})$ $A\subseteq B\Rightarrow \mu (A)\leq \mu (B)$

$(M_{3})$ $\mu (\bar{A})=\mu (A)$, where $\bar{A}$ denotes the closure of $A$

$(M_{4})$ $\mu (conv$ $A)=$ $\mu (A)$,

$(M_{5})$ $\mu (\lambda A+(1-\lambda )B)\leq \lambda \mu (A)+(1-\lambda )\mu
(B)$ for $\lambda \in \left[ 0,1\right] $

$(M_{6})$ If $\left\{ A_{n}\right\} $ is a sequence of closed sets in $M_{X}$
such that $A_{n+1}\subseteq A_{n}$ for $n=1,2,\ldots $ and $%
\lim\limits_{n\rightarrow \infty }\mu (A_{n})=0$, then the following
intersection is nonempty.%
\begin{equation*}
A_{\infty }=\underset{n=1}{\overset{\infty }{\cap }}A_{n}
\end{equation*}

If $(M_{4})$ holds, then $A_{\infty }\in $ Ker$\mu .$ To do this, let $%
\lim\limits_{n\rightarrow \infty }\mu (A_{n})=0.$ As $A_{\infty }\subseteq
A_{n}$ for each $n=0,1,2,...;$ by the monotonicity of $\mu ,$ we obtain 
\begin{equation*}
\mu (A_{\infty })\leq \lim\limits_{n\rightarrow \infty }\mu (A_{n})=0.
\end{equation*}%
So, by $(M_{1})$, we get that $A_{\infty }$ is nonempty and$\ A_{\infty }\in 
$ Ker$\mu $.
\end{definition}

\begin{theorem}[see; \protect\cite{schauder}]
\label{sch}Let $E$ be a closed and convex subset of a Banach space $X$. Then
every compact, continuous map $T:E\rightarrow E$\ has at least one fixed
point.
\end{theorem}

\begin{theorem}[see; \protect\cite{darbo}]
\label{darb}Let $E$ be a nonempty, bounded, closed and convex subset of a
Banach space $X$ and let\ $T:E\rightarrow E$ be a continuous mapping.
Suppose that there exists a constant $k\in \left[ 0,1\right) $ such that%
\begin{equation*}
\mu (T(A))\leq k\mu (A)
\end{equation*}%
for any subset $A$ of $E$, then $T$ has a fixed point.
\end{theorem}

\begin{definition}[see; \protect\cite{berzig}]
\label{berzig}Let $\psi ,\phi :[0,\infty )\rightarrow 
\mathbb{R}
$ be two functions. The pair of functions $\left( \psi ,\phi \right) $ is
said to be a pair of shifting distance function, if the following conditions
hold

$(i)$ for $u,v\in \lbrack 0,\infty )$ if $\psi (u)\leq \phi (v)$, then $%
u\leq v,$

$(ii)$ for $\left\{ u_{k}\right\} ,\left\{ v_{k}\right\} \subset \lbrack
0,\infty )$ with $\underset{k\rightarrow \infty }{lim}u_{k}=\underset{%
k\rightarrow \infty }{lim}v_{k}=w$, if $\psi (u_{k})\leq \phi (v_{k})$ for
all $k\in 
\mathbb{N}
$, then $w=0$.
\end{definition}

\section{Main Result}

Now we will give the definition of a pair of sequences of functions with
shifting distance property and by using this function class, we introduce a
new type of Darbo's fixed point theorem. Afterward, we investigate fixed
point of mapping according to generalized new Darbo's fixed point theorem.
Also we are going to give one interesting example.

\begin{definition}
\label{shift}Let $\psi _{n},\phi _{n}:[0,\infty )\rightarrow 
\mathbb{R}
$ be two sequences of functions. The pair of sequences of functions $\left(
\psi _{n},\phi _{n}\right) $ is said to be a pair of sequences of functions
with shifting distance property which satisfy the following conditions

$(i)$ for $u,v\in \lbrack 0,\infty )$ if $\psi _{n}(u)\leq \phi
_{n}(v)\rightarrow \psi (u)\leq \phi (v)$ uniformly in $n$, then $u\leq v,$

$(ii)$ for $\left\{ u_{k}\right\} ,\left\{ v_{k}\right\} \subset \lbrack
0,\infty )$ with $\underset{k\rightarrow \infty }{lim}u_{k}=\underset{%
k\rightarrow \infty }{lim}v_{k}=w$, if $\psi _{n}(u_{k})\leq \phi
_{n}(v_{k})\rightarrow \psi (u_{k})\leq \phi (v_{k})$ \ for all $k\in 
\mathbb{N}
$, uniformly in $n$ , then $w=0$.
\end{definition}

\begin{definition}
The pair $\left( \psi _{n},\phi _{n}\right) $ is said to be having shifting
distance property if $\left( \psi _{n},\phi _{n}\right) \rightarrow \left(
\psi ,\phi \right) $ uniformly in $n$ and the pair $\left( \psi ,\phi
\right) $ is shifting distance function.
\end{definition}

\begin{lemma}
Let $\psi _{n},\phi _{n}:[0,\infty )\rightarrow R$ be two sequences of
functions. Assume that the sequences of functions hold following conditions

$(i)$ if $\left(\psi _{n}\right)$ upper semi-continuous sequence of
functions and $\psi _{n}\leq \psi _{n+1}$, then $\psi_{n} \rightarrow \psi $ convergent (uniformly in $n$),

$(ii)$ if $\left( \phi _{n}\right) $ lower semi-continuous sequence of
functions and $\phi _{n}\geq \phi _{n+1}$, then $\phi _{n}\rightarrow \phi $
convergent (uniformly in $n$).

Then, $\left( \psi _{n},\phi _{n}\right) $ is called the pair of sequences
of functions having shifting distance property.
\end{lemma}

\begin{proof}
Let $\left( \psi _{n}\right) $ be increasing and bounded by $\psi (u)$, also
let $\left( \phi _{n}\right) $ be decreasing and bounded by $\phi (v)$. By
using condition $(i)$ of Definition \ref{berzig}, we get 
\begin{equation*}
\psi _{1}(u)<\psi _{2}(u)<\cdots <\psi _{n}(u)<\psi (u)\leq \phi (v)<\phi
_{n}(v)<\cdots <\phi _{2}(v)<\phi _{1}(v)\text{.}
\end{equation*}%
Hence we can write for all $n\in 
\mathbb{N}
$ 
\begin{equation}
\psi _{n}(u)\leq \phi _{n}(v) .  \label{1}
\end{equation}%
Taking limit on both sides of (\ref{1}), 
\begin{eqnarray*}
\underset{n\rightarrow \infty }{lim}\psi _{n}(u) &\leq &\underset{%
n\rightarrow \infty }{lim}\phi _{n}(v) \\
\psi (u) &\leq &\phi (v)\Rightarrow u\leq v.
\end{eqnarray*}%
By using condition $(ii)$ of Definition \ref{berzig}, for $\left\{
u_{k}\right\} ,\left\{ v_{k}\right\} \subset \lbrack 0,\infty )$ with $%
\underset{k\rightarrow \infty }{lim}u_{k}=\underset{k\rightarrow \infty }{lim%
}v_{k}=w$ if $\psi _{n}(u_{k})\leq \phi _{n}(v_{k})$ for all $n,k\in 
\mathbb{N}
,$ taking limit and we get%
\begin{eqnarray*}
\underset{n\rightarrow \infty }{lim}\psi _{n}(u_{k}) &\leq &\underset{%
n\rightarrow \infty }{lim}\phi _{n}(v_{k}) \\
\text{ \ \ }\psi (u_{k}) &\leq &\phi (v_{k})\Rightarrow w=0.
\end{eqnarray*}%
That is $\left( \psi _{n},\phi _{n}\right) $ is the pair of sequences of
functions with shifting distance property.
\end{proof}

\begin{theorem}
\label{main}Let $E$ be a nonempty, bounded, closed and convex subset of the
Banach space $X$. Suppose that $T:E\rightarrow E$ is a continuous mapping
such that%
\begin{equation}
\psi _{n}(\mu (TA))\leq \phi _{n}(\mu (A))  \label{2}
\end{equation}%
for any nonempty subset $A\subset $ $E$, where $\mu $ is an arbitrary
measure of noncompatness and $\psi _{n},\phi _{n}:[0,\infty )\rightarrow 
\mathbb{R}
$ be the pair of sequences of functions with shifting distance property.
Then, $T$ has a fixed point in $E.$
\end{theorem}

\begin{proof}
We define a sequence $\left\{ A_{k}\right\} $ such that $A_{0}=A$ and $%
A_{k}=Conv\left( TA_{k-1}\right) $ for all $k\geq 1.$ Then we get 
\begin{eqnarray*}
TA_{0} &=&TA\subseteq A=A_{0} \\
A_{1} &=&Conv\left( TA_{0}\right) \subseteq A=A_{0}.
\end{eqnarray*}%
By repeating process mentioned above, we have 
\begin{equation*}
A_{0}\supseteq A_{1}\supseteq E_{2}\supseteq \cdots \supseteq A_{k}\supseteq
\cdots .
\end{equation*}%
If there exists an integer $k\geq 0$ such that $\mu \left( A_{k}\right) =0$,
then $A_{k}$ is relatively compact and since%
\begin{equation*}
TA_{k}\subseteq Conv\left( TA_{k}\right) =A_{k+1}\subseteq A_{k},
\end{equation*}%
Theorem \ref{sch} implies that $T$ has a fixed point in Schauder's sense on
the set $A_{k}$ for all $k\geq 0.$ Now we assume that $\mu (A_{k})>0$ for
all $k\geq 0.$ By using (\ref{2}) we have%
\begin{eqnarray}
\psi _{n}(\mu (A_{k+1})) &=&\psi _{n}(\mu (Conv\left( TA_{k}\right) ))
\label{3} \\
&=&\psi _{n}(\mu \left( TA_{k}\right) )  \notag \\
&\leq &\phi _{n}(\mu (A_{k})).  \notag
\end{eqnarray}

Let us consider (\ref{2}). Then we obtain that $\left\{ \mu (A_{k})\right\} $
is a decreasing sequence of positive real numbers and there exists $p\geq 0$
such that $\mu (A_{k})\rightarrow p$ as $k\rightarrow \infty .$ By using (%
\ref{3}) and Lemma 3.3, we have 
\begin{equation*}
\psi _{n}(\mu (A_{k+1}))\rightarrow \psi (\mu (A_{k+1})),\text{ uniformly in 
}n
\end{equation*}%
and 
\begin{equation}
\psi (\mu (A_{k+1}))=\psi (\mu (TA_{k})).  \label{5}
\end{equation}%
Also, if $\mu (A_{k})\rightarrow p$ as $k\rightarrow \infty ,$ then $\mu
(A_{k+1})\rightarrow p$ as $k\rightarrow \infty .$ Hence we have%
\begin{eqnarray*}
\lim_{k\rightarrow \infty }\psi (\mu \left( A_{k+1}\right) )
&=&\lim_{k\rightarrow \infty }\psi (\mu \left( TA_{k}\right) )\leq
\lim_{k\rightarrow \infty }\phi (\mu (A_{k})) \\
\psi (p) &\leq &\phi (p).
\end{eqnarray*}

By condition $(ii)$ of Definition \ref{shift}, we get $p=0$. So we have $\mu
(A_{k})\rightarrow 0$ as $k\rightarrow \infty .$ On the other hand, since $%
A_{k+1}\subseteq A_{k},$ $TA_{k}\subseteq A_{k}$ and $\mu (A_{k})\rightarrow
0$ as $k\rightarrow \infty .$ Using $(M_{6})$ of Definition \ref{darbo}, $%
A_{\infty }=\underset{k=1}{\overset{\infty }{\cap }}A_{k}$ is nonempty,
closed, convex, and invariant under $T.$ Hence the mapping $T$ belong to Ker 
$\mu .$Therefore, Schauder's fixed point theorem implies that has a fixed
point in $A_{\infty }\subset $ $A$.
\end{proof}

\begin{example}
\bigskip The under the conditions of \ Definition \ref{shift}, we consider
the following sequence of functions%
\begin{equation*}
\psi _{n}(u)=\frac{2n(1+u)+2u+1}{n+1},\text{ }\phi _{n}(v)=\frac{n(2+v)+1}{n}%
.
\end{equation*}

It is clear that  $\psi _{n}(u)\leq \phi _{n}(v),$ for all $n\in
\mathbb{N}
$ and $u,v\in \lbrack 0,\infty ).$ It is easy to see that the pairs $\left(
\psi _{n},\phi _{n}\right) \rightarrow \left( \psi ,\phi \right) $ are
shifting distance function. To see this, we have
\begin{equation*}
\lim_{n\rightarrow \infty }\frac{2n(1+u)+2u+1}{n+1}=2+2u\leq
2+v=\lim_{n\rightarrow \infty }\frac{n(2+v)+1}{n}.
\end{equation*}

Therefore $\left( \psi ,\phi \right) $ is shifting distance functions.

Now we suppose that $u=\mu (TA)$ and $v=\mu (A).$ Since%
\begin{equation*}
\frac{2n(1+\mu (TA))+2\mu (TA)+1}{n+1}\leq \frac{n(2+\mu (A))+1}{n},
\end{equation*}

we have
\begin{equation}
2\mu (TA)-\mu (A)\leq \frac{2n+1}{n(n+1)}.  \label{6}
\end{equation}

if limit goes to infinity in (\ref{6}), we obtain
\begin{equation*}
2\mu (TA)-\mu (A)\leq 0.
\end{equation*}%
As a result,
\begin{eqnarray*}
2\mu (TA) &\leq &\mu (A) \\
\mu (TA) &\leq &\frac{1}{2}\mu (A).
\end{eqnarray*}%
Therefore, according to condition of Darbo's fixed point theorem, $T$ has a
fixed point under continuous sequences of functions. Hence this completes the
proof.
\end{example}

If we take $\psi _{n}=I_{n}$ such that $\underset{n\rightarrow \infty }{\lim 
}I_{n}=I$ uniformly convergence for all $n\in 
\mathbb{N}
$ \ in Theorem \ref{main}, we obtain the following result.

\begin{corollary}
Let $E$ be a nonempty, bounded, closed and convex subset of the Banach space 
$X$. Suppose that $T:E\rightarrow E$ is a continuous function such that 
\begin{equation*}
I_{n}\left( \mu (TA)\right) \leq \phi _{n}(\mu (A))
\end{equation*}%
for any nonempty subset of $A\subset E$, where $\mu $ is an arbitrary
measure of noncompactness and $\phi _{n}:[0,\infty )\rightarrow 
\mathbb{R}
$ be a sequence of function such that

$(a)$ for $u,v\in \lbrack 0,\infty )$ if $I_{n}\left( u\right) \leq \phi
_{n}(v)$, then $u\leq v,$

$(b)$ for $\left\{ u_{k}\right\} ,\left\{ v_{k}\right\} \subset \lbrack
0,\infty )$ with $\underset{k\rightarrow \infty }{lim}u_{k}=\underset{%
k\rightarrow \infty }{lim}v_{k}=w$, if $I_{n}\left( u_{k}\right) \leq \phi
_{n}(v_{k})$ for all $n,k\in 
\mathbb{N}
$, then $w=0$.

Then, $T$ has a fixed point in $E.$
\end{corollary}

\begin{corollary}
\label{cor}Let $E$ be a nonempty, bounded, closed and convex subset of the
Banach space $X$. Suppose that $T:E\rightarrow E$ is a continuous mapping
such that%
\begin{equation}
\psi _{n}\left( \mu (TA)\right) \leq \psi _{n}\left( \mu (A)\right) -\phi
_{n}(\mu (A))  \label{4}
\end{equation}%
for any nonempty subset of $A\subset E$, where $\mu $ is an arbitrary
measure of noncompactness and $\psi _{n},\phi _{n}:[0,\infty )\rightarrow 
\mathbb{R}
^{+}$ be a pair having shifting distance property. Also the pair $\left(
\psi ,\phi \right) $ is two nondecreasing and continuous functions
satisfying $\psi (t)=$ $\phi (t)$ if and only if $t=0$. Then, $T$ has a
fixed point in $E.$
\end{corollary}

\begin{proof}
Assume that (\ref{4}) holds. If by taking limit on (\ref{4}), we get 
\begin{equation}
\psi \left( \mu (TA)\right) \leq \psi \left( \mu (A)\right) -\phi (\mu (A)).
\label{7}
\end{equation}

Besides, by using hypothesis in statement, we suppose that $\psi \left( \mu
(A)\right) =\phi (\mu (A)).$ Then we get $\mu (A)=0.$ Under the conditions
of Theorem \ref{main} , $A$ is relatively compact and then Theorem \ref{sch}
implies that $T$ has a fixed point in $E$. Conversely, we suppose that $\mu
(A)=0.$ Then in (\ref{7}) $\psi \left( \mu (A)\right) =\phi (\mu (A)).$
Since $\mu (A)=0,$ it is clear that $A$ is relatively compact. Hence using
Theorem \ref{sch}, again, we say that $T$ has a fixed point in $E.$ Also
since $\left( \psi ,\phi \right) \in 
\mathbb{R}
^{+},$ $\mu (TA)=0.$ So by repeating the conditions of Theorem \ref{main},
we obtain that $T$ belong to Ker $\mu $. \ As a result, mapping $T$ has a
fixed point in $A_{\infty }\subset A.$
\end{proof}

\begin{corollary}
Let $E$ be a nonempty, bounded, closed and convex subset of a Banach space $%
X $ and let\ $T:E\rightarrow E$ be a continuous mapping. Suppose that there
exists a constant $k\in \lbrack 0,1)$ such that%
\begin{equation*}
\mu (TA)\leq k\mu (A),
\end{equation*}%
for any subset $A\subset E$, then $T$ has a fixed point.
\end{corollary}

\begin{proof}
Taking $\psi _{n}(t)=I_{n}$ and $\phi _{n}(t)=kI_{n}$ such that $%
I_{n}\rightarrow I$ uniformly $n$ in Theorem \ref{main}, we get Darbo's
fixed point theorem, where $k\in \lbrack 0,1).$
\end{proof}
\noindent\textbf{Acknowledgement.} This work was supported by the Ahi Evran
University Scientific Research Projects Coordination Unit. Project Number:
RKT. A3.17.001.


\begin{thebibliography}{99}
\bibitem{schauder} R. Agarwal, M. Meehan, D. O'Regan, Fixed point theory and
applications, Cambridge University, Cambridge, 2004.

\bibitem{banach} S. Banach, Sur les operations dans les ensembles abstraits
et leur application aux equations integrales, Publie dans Fund. Math.3,
p.133-181, 1922.

\bibitem{banas1} J. Banas, Measures of noncompactness in the study of
solutions of nonlinear differential and integral equations, J. Cent. Eur. J.
Math., 2012. doi: 10.2478/s11533-012-0120-9.

\bibitem{darbo} J. Banas, K. Goebel, Measure of noncompactness in Banach
spaces, Lecture Notes in Pure and Applied Mathematics, 60 (1980), Dekker,
New York.

\bibitem{banas2} J. Banas, T. Zajac, Solvability of a functional integral
equation of fractional order in the class of functions having Limits at
infinity, Nonlinear Anal., 71(11): 5491--5500, 2009.

\bibitem{berzig} M. Berzig, Generalization of the Banach Contraction
Principle, 2013. arXiv:1310.0995 [math.CA]

\bibitem{bouzara} N.E.H. Bouzara, V. Karakaya, On different type of fixed
point theorem for multivalued mappings via measure of noncompactness, Adv.
Oper. Theory 3 (2018), no. 2, 326--336.

\bibitem{darbofixed} G., Darbo, Punti uniti in trasformazioni a codominio
non compatto, Rend. Sem. Mat.Univ. Padova, 24, 84-92, 1955.

\bibitem{dutta} P.N. Dutta, B.S. Choudhury, A generalization of contraction
principle in metric spaces, Fixed Point Theory Appl. 2008, 1--8, 2008.

\bibitem{karakaya} V. Karakaya, Y. Atalan, K. Do\u{g}an, N.EH. Bouzara, Some
fixed point results for a new three steps iteration process in Banach
spaces, Fixed Point Theory and Applications, 18(2):\ 625-640, 2017.

\bibitem{khan} M.S. Khan, M. Swaleh, S. Sessa, Fixed point theorems by
altering distances between the points, Bull. Aust. Math. Soc. 30, 1--9, 1984.

\bibitem{kuro} K., Kuratowski, Sur les espaces complets, Fund. Math.15,
301-309, 1930.

\bibitem{mursaleen} M. Mursaleen, S.A. Mohiuddine, Applications of measures
of noncompactness to the infinite system of differential equations in lp
spaces, Nonlinear Anal. 75(4): 2111--2115, 2012.

\bibitem{rhoades} B.E. Rhoades, Some theorems on weakly contractive maps,
Nonlinear Anal. 47, 2683--2693, 2001.

\bibitem{samadi} A. Samadi and M. B. Ghaemi, An extension of Darbo's theorem
and its application, Abstract and Applied Analysis, Vol. 2014, Article ID
852324, 11 pages.

\bibitem{derya} D. Sekman, N.E.H. Bouzara, V. Karakaya, n-tuplet fixed
points of multivalued mappings via measure of noncompactness, Communications
in Optimization Theory, Vol. 2017, Article ID 24, pp. 1-13, 2017.
\end{thebibliography}
\end{document}